\documentclass[a4paper,12pt]{article}
    \usepackage[top=2cm,bottom=2cm,left=2.5cm,right=2.5cm]{geometry}
    \usepackage{cite, amsmath, amssymb}
    \usepackage[margin=0.5cm,%
                font=small,%
                format=hang,%
                labelsep=period,%
                labelfont=bf]{caption}
\usepackage{amsmath}
\usepackage{amsfonts}
\usepackage{amssymb}
\usepackage{cite}
\usepackage{amsthm}
\usepackage{makeidx}
\usepackage{graphicx}
\usepackage{mathrsfs,xcolor,multicol}
\usepackage{enumerate}
\usepackage{blkarray}
\usepackage{bm}
\usepackage{setspace}
\usepackage{float}
\usepackage{authblk}
\usepackage{multirow}
\usepackage{booktabs}
\usepackage[ruled,vlined]{algorithm2e}
\usepackage{blkarray}
\usepackage{authblk}
\usepackage{multirow}
\usepackage{lineno}

\usepackage{blkarray, bigstrut}
\usepackage{authblk}
\usepackage{multirow}

\usepackage{fancyhdr}
\usepackage{lastpage}

\usepackage{amsmath, amssymb, booktabs, lastpage,}
\usepackage{multicol}
\usepackage{hyperref}
\usepackage{blkarray}
\usepackage{empheq} 

\numberwithin{table}{section}
\numberwithin{equation}{section}

\theoremstyle{plain}
\newtheorem{theorem}{Theorem}[section]
\newtheorem{proposition}[theorem]{Proposition}
\newtheorem{definition}[theorem]{Definition}

\newtheorem{example}[theorem]{Example}

\newtheorem{remark}[theorem]{Remark}

\usepackage{authblk}
\author[1]{ \textbf{Bryan S. Hernandez}}
\author[2]{ \textbf{Patrick Vincent N. Lubenia}}
\author[2,3,4]{\textbf{Eduardo R. Mendoza}}

\affil[1]{\small \textit{Institute of Mathematics, University of the Philippines Diliman, Quezon City 1101, Philippines}}
\affil[2]{\small \textit{Systems and Computational Biology Research Unit, Center for Natural Sciences and Environmental Research, Manila 0922, Philippines}}
\affil[3]{\small \textit{Mathematics and Statistics Department, De La Salle University, Manila  0922, Philippines}}
\affil[4]{\small \textit{Max Planck Institute of Biochemistry, Martinsried near Munich 82152, Germany}}
\affil[*]{Email addresses: \texttt{bshernandez@up.edu.ph},
\texttt{pnlubenia@upd.edu.ph}, \texttt{eduardo.mendoza@dlsu.edu.ph}}

\title{\textbf{Equilibria Decomposition-Based Comparison of Reaction Networks of Wnt Signaling}}
\date{}

\begin{document}
\maketitle
\begin{abstract} 
The Wnt signaling pathway plays a critical role in various biochemical processes, including embryonic development, tissue homeostasis, and cancer progression. In this paper, we conduct a comparative analysis of $\beta$-catenin-dependent Wnt signaling reaction networks, which we refer to as the Feinberg, Schmitz, and MacLean models, based on the previous study by MacLean et al. (PNAS USA 2015).
Our analysis is based on the (unique) finest independent decomposition (FID) of each reaction network and our comparative techniques include equilibria parametrizations (EP) and the newly developed methods of Common Reactions Equilibria (CORE) analysis and Concordance Profile (CP) analysis.
Our investigation yields three interesting results concerning the equilibria sets of these models. Firstly, we explore the concept of absolute concentration robustness (ACR), wherein a system exhibits ACR in a specific species if the equilibrium value for that species is the same for any positive equilibrium. Through ACR analysis employing FID and EP, we observe that both the Schmitz and MacLean models lack ACR, whereas the Feinberg model demonstrates ACR in a single species.
Second, our analyses using FID and CORE reveal important relationships within the equilibria sets of the augmented Schmitz and MacLean models.
Furthermore, FID and CORE identify the lack of a substantial relationship between the equilibria sets of the Feinberg and MacLean models.
Hence, these methods detect subtle differences between the Feinberg and MacLean models and also between the Schmitz and MacLean models, which are not evident in the standard reaction network analysis.
Finally, based on the concordance levels, CP analysis indicates that the MacLean and Schmitz models are more similar than the MacLean and Feinberg models.
	
\end{abstract}

\thispagestyle{empty}
\section{Introduction}
\label{sec:1}

The Wnt signaling pathway plays a crucial role in numerous biochemical processes, including embryonic development, tissue homeostasis, and cancer progression
\cite{Liu,Patel:Wnt,Sharma:Wnt,Silveira:Wnt,Yang2016}.
Parameter-free analysis through reaction networks has been used as a tool for studying the complexities of Wnt signaling pathways \cite{Maclean} primarily to investigate the occurrence of bistability, i.e., the network has the capacity to admit two stable positive equilibria
with the same conserved quantities
\cite{dickenstein2019multistationarity}.

Building upon previous research efforts, particularly those of MacLean et al. \cite{Maclean}, we analyze and compare $\beta$-catenin-dependent Wnt signaling models.
In particular, we utilize methods of finest independent decompositions (FIDs) \cite{Hernandezetal2022,HEDC2021} and equilibria parametrizations (EP) \cite{HernandezetalPCOMP2023,JMP2019:parametrization}. Importantly, motivated by Wnt signaling, we introduce two novel approaches, which we call Common Reactions Equilibria (CORE) analysis and Concordance Profile (CP) analysis, to reveal kinetic and structural relationships among reaction networks on the basis of their sets of positive equilibria.

The Lee model \cite{Lee} focuses mainly on elucidating the formation of the destruction complex from its individual components and how its subsequent ability to degrade $\beta$-catenin is influenced by the presence or absence of an external Wnt stimulus. This model assumes a uniform distribution of all species throughout the cell, without distinguishing between the nucleus and the cytoplasm.
Feinberg \cite{FeinbergBook2019} introduced modifications to the Lee model by consolidating a complex of three species into a single species, making a reaction reversible, and eliminating another species from the network. The model was used by Feinberg to explore the concept of the ``hidden'' concordance property in a reaction network \cite{FeinbergBook2019}.
Next, the Schmitz model \cite{Schmitz} investigates the impact of shuttling of $\beta$-catenin and the destruction complex between the cytoplasm and the nucleus on the binding of the T-cell factor to $\beta$-catenin within the nucleus.
Finally, the MacLean model \cite{Maclean} focuses on both degradation of $\beta$-catenin and shuttling between the cytoplasm and the nucleus, potentially serving as a mechanism for governing bistability in the pathway \cite{Maclean}.

In our recent study \cite{HLM2024embeddingbased}, we showed strong similarity between the Lee and the Feinberg models. Furthermore, the Lee model is mono-stationary while the three remaining models are multi-stationary. 
Hence, we focus on the analysis of the latter three multistationary models: the Feinberg, Schmitz, and MacLean models. The reaction networks of the models are given in Tables \ref{tab:reactionsNSandNM} and \ref{tab:reactionsNFandNM}.

Using the algorithm of Hernandez et al. \cite{HernandezetalPCOMP2023}, which utilizes FIDs, we determine equilibria parametrizations of the networks and infer the network species with absolute concentration robustness (ACR) \cite{SHFE2010}:
the Schmitz and MacLean networks lack ACR while
the Feinberg network has ACR in a single species. This means that the positive equilibrium of the said species of the Feinberg network does not depend on any initial concentrations of the other species. On the other hand, the value of the positive equilibrium value of any species of the Schmitz and MacLean networks are always dependent on the initial concentrations of some species.

Importantly, our analyses using FID and CORE reveal essential relationships within the equilibria sets of the augmented Schmitz and MacLean models. Furthermore, these two methods show the absence of a significant relationship between the equilibria sets of the Feinberg and MacLean models. These findings are not evident in the standard reaction network analysis.
We also analyze the level of concordance of the networks. 
According to the concordance levels, CP analysis suggests that the MacLean and Schmitz networks exhibit greater similarity compared to the MacLean and Feinberg networks.

By employing a computational approach, our study contributes to the ongoing efforts in investigating complexities of biochemical pathways and offers new perspectives for comparing biochemical reaction networks.

\section{Comparison of basic structural and kinetic properties of the networks}
\label{section:structural:kinetic:properties}

Various reports were generated using the Windows application CRNToolbox \cite{FeinbergToolbox} to complete the overview of the basic properties of the networks, and their results are collected in Table \ref{tab:CRNToolbox}. The kinetic properties include both purely kinetic, i.e., those that remain invariant under dynamical equivalence, and structo-kinetic, i.e., those which may vary. Non-degeneracy of an equilibrium is an example of the former class, while mono-/multi-stationarity is an example of the latter type.
As seen from the table, we can verify the coincidence of the basic properties among the Feinberg, Schmitz, and MacLean models, which we denote as $\mathscr{N}_F$, $\mathscr{N}_S$, and $\mathscr{N}_M$, respectively.

\begin{table}[ht!]
	\centering
	\caption{CRNToolbox results of the Feinberg ($\mathscr{N}_F$), Schmitz ($\mathscr{N}_S$) and MacLean ($\mathscr{N}_M$) Wnt signaling models}
	\label{tab:CRNToolbox}
	\begin{tabular}{|l|c|c|c|}
		\hline
		\multicolumn{1}{|c|}{Property} &  $\mathscr{N}_F$ & $\mathscr{N}_S$ & $\mathscr{N}_M$ \\
		\hline
            Conservative & No & No & No \\
		\hline
            Positive dependent & Yes & Yes & Yes \\
		\hline
            Existence of degenerate  & \multirow{2}{*}{\shortstack{Yes}} & \multirow{2}{*}{\shortstack{Yes}} & \multirow{2}{*}{\shortstack{Yes}} \\
            equilibrium & & & \\
		\hline
            Non-degenerate network & Yes & Yes & Yes \\
		\hline
            Injective
            & No & No & No \\
            \hline
            Multi-stationary & Yes & Yes & Yes \\
            \hline
	\end{tabular}
\end{table}

\section{Comparison of the FIDs of the Feinberg, MacLean, and Schmitz models}
\label{section:comparison:three:models}

In this section, we first present a method for comparative analysis using the FIDs of reaction networks. We then compute the FIDs of the three networks, which turn out to differ significantly. In subsequent sections, we use this variation in FIDs to infer interesting relationships between networks.

\subsection{A novel structural property: the network's finest independent decomposition}
\label{structural:property:FID}

The concept of FID was introduced by B. Hernandez and R. De la Cruz in \cite{HEDC2021}. Since some networks have only the trivial decomposition as an independent decomposition, they provided a topological characterization for the existence of a non-trivial FID. Furthermore, they developed an algorithm to compute the said decomposition. The uniqueness of the FID was shown in \cite{Hernandezetal2022}, establishing the FID as a novel network property. The subnetworks from the FID comprise the smallest ``building blocks'' of a network.

The significance of FID for any kinetics on a network derives from Feinberg's Decomposition Theorem \cite{Feinberg1987stability,FeinbergBook2019}: for any independent decomposition, the positive equilibria set of the whole network is the intersection of the positive equilibria sets of the independent subnetworks (Theorem \ref{theorem:independent} in the Appendix in this paper). 

\begin{remark}
The positive equilibria of the FID subnetworks are the smallest ``building blocks'' of the equilibria of the whole network.  
\end{remark}

\subsection{Computation and comparison of the FIDs of the three networks}

The following lists and tables collect the information about the FIDs of the three networks, providing the basis for our subsequent analysis.

Using the algorithm of Hernandez and De la Cruz \cite{HEDC2021}, the FID of $\mathscr{N}_F$ has eight subnetworks:

    \noindent
    $\mathscr{N}_{F,1} = \{ R_1, R_4, R_5, R_{12}, R_{38}, R_{45}, R_{46}\}$ \\
    $\mathscr{N}_{F,2} = \{ R_{14}, R_{15} \}$ \\
    $\mathscr{N}_{F,3} = \{ R_{18}, R_{19} \}$ \\
    $\mathscr{N}_{F,4} = \{ R_{43}, R_{44} \}$ \\
    $\mathscr{N}_{F,5} = \{ R_{47}, R_{48} \}$ \\
    $\mathscr{N}_{F,6} = \{ R_{49}, R_{50} \}$ \\
    $\mathscr{N}_{F,7} = \{ R_{51}, R_{52} \}$ \\
    $\mathscr{N}_{F,8} = \{ R_{53}, \dots, R_{56} \}$.

Table \ref{tab:FALNetworkNumbers} presents the network numbers of the subnetworks from the FID of $\mathscr{N}_F$.

\begin{table}[ht!]
    \centering
    \caption{Network numbers of the Feinberg model ($\mathscr{N}_F$) and the subnetworks of its FID}
    \label{tab:FALNetworkNumbers}
    \begin{tabular}{lccccccccc}
	\hline
        Network numbers & $\mathscr{N}_F$ & $\mathscr{N}_{F,1}$ & $\mathscr{N}_{F,2}$ & $\mathscr{N}_{F,3}$ & $\mathscr{N}_{F,4}$ & $\mathscr{N}_{F,5}$ & $\mathscr{N}_{F,6}$ & $\mathscr{N}_{F,7}$ & $\mathscr{N}_{F,8}$ \\
	\hline
	Species                         & 15 & 5 & 2 & 2 & 3 & 1 & 3 & 3 & 4 \\
        Complexes                       & 21 & 7 & 2 & 2 & 2 & 2 & 2 & 2 & 5 \\
        Reactant complexes              & 19 & 6 & 2 & 2 & 2 & 2 & 2 & 2 & 4 \\
        Reversible reactions            & 9  & 2 & 1 & 1 & 1 & 1 & 1 & 1 & 1 \\
        Irreversible reactions          & 5  & 3 & 0 & 0 & 0 & 0 & 0 & 0 & 2 \\
        Reactions                       & 23 & 7 & 2 & 2 & 2 & 2 & 2 & 2 & 4 \\
        Linkage classes (LC)                 & 7 & 2 & 1 & 1 & 1 & 1 & 1 & 1 & 2 \\
        Strong LC          & 12 & 5 & 1 & 1 & 1 & 1 & 1 & 1 & 4 \\
        Terminal strong LC & 7  & 2 & 1 & 1 & 1 & 1 & 1 & 1 & 2 \\
        Rank                            & 12 & 4 & 1 & 1 & 1 & 1 & 1 & 1 & 2 \\
        Deficiency                      & 2  & 1 & 0 & 0 & 0 & 0 & 0 & 0 & 1 \\
	\hline
    \end{tabular}
\end{table}

On the other hand, $\mathscr{N}_M$ has the following FID. The corresponding network numbers are shown in Table \ref{tab:MacLeanFID}:

\allowdisplaybreaks
\begin{multicols}{2}
\noindent
\begin{align*}
	& \mathscr{N}_{M,1} = \{ R_1, \ldots, R_7, R_{36}, \ldots, R_{39} \} \\
	& \mathscr{N}_{M,2} = \{ R_8, R_9 \} \\
	& \mathscr{N}_{M,3} = \{ R_{18}, R_{19} \} \\
	& \mathscr{N}_{M,4} = \{ R_{20}, R_{21} \} \\
	& \mathscr{N}_{M,5} = \{ R_{22}, R_{23} \} \\
	& \mathscr{N}_{M,6} = \{ R_{24}, \ldots, R_{29} \} \\
	& \mathscr{N}_{M,7} = \{ R_{30}, \ldots, R_{35} \}.
\end{align*}
\end{multicols}

\begin{table}[ht!]
	\centering
	\caption{Network numbers of the MacLean model ($\mathscr{N}_M$) and the subnetworks of its FID}
	\label{tab:MacLeanFID}
	\begin{tabular}{lcccccccc}
		\hline
		Network numbers & $\mathscr{N}_M$ & $\mathscr{N}_{M,1}$ & $\mathscr{N}_{M,2}$ & $\mathscr{N}_{M,3}$ & $\mathscr{N}_{M,4}$ & $\mathscr{N}_{M,5}$ & $\mathscr{N}_{M,6}$ & $\mathscr{N}_{M,7}$ \\
		\hline
		Species & 19 & 6 & 3 & 2 & 2 & 2 & 6 & 6 \\
		Complexes & 28 & 9 & 2 & 2 & 2 & 2 & 6 & 6 \\
		Reactant complexes & 22 & 7 & 2 & 2 & 2 & 2 & 4 & 4 \\
		Reversible reactions & 12 & 3 & 1 & 1 & 1 & 1 & 2 & 2 \\
		Irreversible reactions & 7 & 5 & 0 & 0 & 0 & 0 & 2 & 2 \\
		Reactions & 31 & 11 & 2 & 2 & 2 & 2 & 6 & 6 \\
		Linkage classes (LC) & 10 & 3 & 1 & 1 & 1 & 1 & 2 & 2 \\
		Strong LC & 16 & 5 & 1 & 1 & 1 & 1 & 4 & 4 \\
		Terminal strong LC & 10 & 3 & 1 & 1 & 1 & 1 & 2 & 2 \\
		Rank & 14 & 4 & 1 & 1 & 1 & 1 & 3 & 3 \\
		Deficiency & 4 & 2 & 0 & 0 & 0 & 0 & 1 & 1 \\
		\hline
	\end{tabular}
\end{table}

Finally, the FID of the Schmitz network $\mathscr{N}_S$ is as follows (Table \ref{tab:schmitzFID} presents the network numbers):

\begin{multicols}{2}
\noindent
\begin{align*}
	& \mathscr{N}_{S,1} = \{ R_1, \ldots, R_7, R_{10}, \ldots, R_{13} \} \\
	& \mathscr{N}_{S,2} = \{ R_8, R_9 \} \\
	& \mathscr{N}_{S,3} = \{ R_{14}, R_{15} \} \\
	& \mathscr{N}_{S,4} = \{ R_{16}, R_{17} \}.
\end{align*}
\end{multicols}

\begin{table}[ht!]
	\centering
	\caption{Network numbers of the Schmitz model ($\mathscr{N}_S$) and the subnetworks of its FID}
	\label{tab:schmitzFID}
	\begin{tabular}{lccccc}
		\hline
		Network numbers & $\mathscr{N}_S$ & $\mathscr{N}_{S,1}$ & $\mathscr{N}_{S,2}$ & $\mathscr{N}_{S,3}$ & $\mathscr{N}_{S,4}$ \\
		\hline
		Species & 11 & 8 & 3 & 2 & 2 \\
		Complexes & 16 & 11 & 2 & 2 & 2 \\
		Reactant complexes & 14 & 9 & 2 & 2 & 2 \\
		Reversible reactions & 6 & 3 & 1 & 1 & 1 \\
		Irreversible reactions & 5 & 5 & 0 & 0 & 0 \\
		Reactions & 17 & 11 & 2 & 2 & 2 \\
		Linkage classes (LC) & 5 & 3 & 1 & 1 & 1 \\
		Strong LC & 10 & 8 & 1 & 1 & 1 \\
		Terminal strong LC & 5 & 3 & 1 & 1 & 1 \\
		Rank & 9 & 6 & 1 & 1 & 1 \\
		Deficiency & 2 & 2 & 0 & 0 & 0 \\
		\hline
	\end{tabular}
\end{table}

We note the wide variation among the subnetworks of the three networks. Only the subnetworks $\mathscr{N}_{M,2}$ and $\mathscr{N}_{S,2}$ coincide. This variation is, of course, due to the FIDs being directly based on the (different) reaction sets. The FID is, thus, the only structural property that differentiates the three multi-stationary networks so far. Though FID is a network property that shows big differences between the three models, FID also serves as the basis of the other techniques (EP, CORE, and CP) as we present in the subsequent sections.

\section{Equilibria parametrizations}
\label{equilibria:parametrizations}

In this section, we use the algorithm of Hernandez et al. \cite{HernandezetalPCOMP2023} to compute the parametrization of positive equilibria of the Schmitz and Feinberg networks, both of which are endowed with mass action kinetics. The algorithm combines the use of FIDs \cite{Hernandezetal2022,HEDC2021} and an earlier method by Johnston et al. \cite{JMP2019:parametrization} for equilibria parametrization (which was previously applied to the MacLean network, referred to as ``shuttled Wnt'' in \cite{JMP2019:parametrization}). Analysis of parametrized equilibria allows us to determine the presence of ACR species in a network. This concept of ACR was introduced by Shinar and Feinberg in the Science journal \cite{SHFE2010}, where they described ACR as the capacity of a species to have the same value at every positive steady state of the system. Generally speaking, in ACR, different sets of initial conditions yield the same steady state value for the species.

\subsection{A brief review of the equilibria parametrization algorithm via network decomposition}

The following summary of the parameterization algorithm for positive equilibria of Hernandez et al. \cite{HernandezetalPCOMP2023} shows the significance of FIDs in the procedure.

Utilizing the Feinberg Decomposition Theorem \cite{Feinberg1987stability,FeinbergBook2019}, which states that {\it{the set of positive equilibria of the whole network is equal to the intersection of the sets of positive equilibria of its stoichiometrically-independent subnetworks}}, we first break the CRN into its smallest independent subnetworks (under the FID). This allows us to compute more easily the parametrized positive equilibria of each subnetwork. We then merge these results to obtain the parametrized positive equilibria of the entire network. The computed positive equilibria of the species {\it{common}} to some subnetworks are equated to each other to get the positive equilibria parametrization of the whole network.

As an illustration, suppose that a given network $\mathscr{N}$ (with three species $X_1$, $X_2$, and $X_3$) has two smaller independent subnetworks under the FID, say, $\mathscr{N}_1$ (with species $X_1$ and $X_2$) and $\mathscr{N}_2$ (with species $X_2$ and $X_3$). We compute the parametrized positive equilibria of each subnetwork independently. The parametrized positive equilibria of $\mathscr{N}$ contain the computed parametrization of $x_1$ (from subnetwork $\mathscr{N}_1$), $x_3$ (from subnetwork $\mathscr{N}_2$), and the solution to $x_2$ after equating the computed parametrizations from subnetworks $\mathscr{N}_1$ and $\mathscr{N}_2$.

\subsection{Equilibria parametrization of the Schmitz system}

Following the method outlined above, we are able to compute a parametrization of the positive equilibria of the Schmitz network ($\mathscr{N}_S$) as follows:
\allowdisplaybreaks
\begin{align*}
    a_{1}&:=y_a=\dfrac{\sigma_1(k_5+k_{10})}{k_4 k_{10}}\\
    a_{2}&:=y_i=\dfrac{k_{14}\sigma_1(k_{5}+k_{10})}{k_4 k_{10} k_{15}}\\
        a_{3}&:=y_{an}=\dfrac{\sigma_2(k_{7}+k_{11})}{k_6 k_{11}}\\
    a_{4}&:=x=\dfrac{k_1(k_3+\sigma_2)}{k_2 \sigma_2 + k_3 \sigma_1 + \sigma_1 \sigma_2}\\
    a_{5}&:=x_{n}=\dfrac{k_1 k_2}{k_2 \sigma_2 + k_3 \sigma_1 + \sigma_1 \sigma_2}\\
    a_{6}&:=t=\tau_2\\
    a_{7}&:=c_{XT}=\dfrac{k_1 k_2 k_{8}\tau_2}{k_{9}(k_2 \sigma_2 + k_3 \sigma_1 + \sigma_1 \sigma_2)}\\
    a_{8}&:=c_{XY}=\dfrac{k_1 \sigma_1(k_3+\sigma_2)}{k_{10}(k_2 \sigma_2 + k_3 \sigma_1 + \sigma_1 \sigma_2)}\\
    a_{9}&:=c_{XYn}=\dfrac{k_1 k_2\sigma_2}{k_{11}(k_2 \sigma_2 + k_3 \sigma_1 + \sigma_1 \sigma_2)}\\
        a_{10}&:=x_{P}=\dfrac{k_1 \sigma_1(k_3+\sigma_2)}{k_{12}(k_2 \sigma_2 + k_3 \sigma_1 + \sigma_1 \sigma_2)}\\
    a_{11}&:=x_{pn}=\dfrac{k_1 k_2\sigma_2}{k_{13}(k_2 \sigma_2 + k_3 \sigma_1 + \sigma_1 \sigma_2)}
\end{align*}
where $\sigma_2=\dfrac{k_{16} k_6 k_{11} (k_5+k_{10}) \sigma_1}{k_{17} k_4 k_{10} (k_{7}+k_{11})}$ and $\sigma_1,\tau_2>0$.

\subsection{Equilibria parametrization of the Feinberg system}
\label{eq:param:FAL}

A positive equilibria parametrization of the Feinberg network ($\mathscr{N}_F$) is as follows:
\allowdisplaybreaks
\begin{align*}
a_1 &= \dfrac{a_2 k_{15}}{k_{14}}\\
a_2 &= \dfrac{k_{55} \sigma_2 a_{23}}{k_{54} (k_{53}+\sigma_2)}\\
a_4 &= \dfrac{k_1 k_{14} (k_5 + k_{45})}{k_{38} k_5 k_{14} + k_{38} k_{14} k_{45} + a_2 k_4 k_{15} k_{45}}\\
a_6 &= \dfrac{a_7 k_{50} (k_{38} k_5 k_{14} + k_{38} k_{14} k_{45} + a_2 k_4 k_{15} k_{45})}{k_1 k_{14} k_{49} (k_5 + k_{45})} \\
a_8 &= \dfrac{a_2 k_1 k_4 k_{15}}{k_{38} k_5 k_{14} + k_{38} k_{14} k_{45} + a_2 k_4 k_{15} k_{45}}\\
a_{10} &= \dfrac{a_2 k_1 k_4 k_{15} k_{45}}{k_{12} (k_{38} k_5 k_{14} + k_{38} k_{14} k_{45} + a_2 k_4 k_{15} k_{45})}\\
a_{12} &= \dfrac{a_{13} k_{19}}{k_{18}}\\
a_{13} &= \dfrac{k_{54}}{\sigma_2}\\
a_{24} &= \dfrac{a_{27} k_{52} (k_{38} k_5 k_{14} + k_{38} k_{14} k_{45} + a_2 k_4 k_{15} k_{45})}{k_1 k_{14} k_{51} (k_5 + k_{45})}\\
a_{25} &= \dfrac{a_2 k_1 k_4 k_{15} k_{45}}{k_{46} (k_{38} k_5 k_{14} + k_{38} k_{14} k_{45} + a_2 k_4 k_{15} k_{45})}\\
a_{26} &= \dfrac{k_{47}}{k_{48}}\\
a_{27} &= \dfrac{a_{23} k_1 k_{14} k_{44} k_{48} k_{51} (k_5 + k_{45})}{k_{43} k_{47} k_{52} (k_{38} k_5 k_{14} + k_{38}k_{14} k_{45} + a_2 k_4 k_{15} k_{45})}\\
a_{28} &= \dfrac{a_{23} k_{53} k_{55}}{k_{56} (k_{53}+\sigma_2)}
\end{align*}
where $\sigma_2,a_7,a_{23}>0$.

\subsection{ACR in the multi-stationary Wnt signaling systems}

The following proposition describes ACR in the three multi-stationary systems under consideration.

\begin{proposition}
The following statements describe the ACR properties of the MacLean, Schmitz, and Feinberg systems:
\begin{enumerate}
    \item[i.] The MacLean and Schmitz mass action systems lack ACR in any species.
    \item[ii.] In the Feinberg mass action system, only $A_{26}$ (i.e., axin) has ACR.
\end{enumerate}
\end{proposition}

\begin{proof}
        To show $(i)$, we reproduce here the equilibria parametrization of Johnston et al. for $\mathscr{N}_M$:
    \allowdisplaybreaks
    \begin{align*}
a_1 &:= y_a = \frac{K_1}{K_2}\\
a_2 &:= y_i = \frac{k_{23}}{k_{22}} \frac{k_{28}+k_{29}}{k_{27}} \frac{\tau_{13}}{\tau_{12}}\\
a_{3} &:= y_{an} = \frac{K_4}{K_3}\\
a_{4} &:=x = \frac{K_2}{K_6}\\
a_{5} &:=x_{n} = \frac{K_3}{K_6}\\
a_{6} &:=t = d_{12}\\
a_{7} &:=c_{XT} = \frac{k_8}{k_9} \frac{K_3}{K_6} d_{12}\\
a_{8} &:= c_{XY} = \frac{K_5}{K_6}\\
a_{9} &:=c_{XYn} = \frac{K_7}{K_6}\\
a_{12} &:= d_i = \frac{k_{19}}{k_{18}} \frac{k_{21}}{k_{20}} \frac{k_{25}+k_{26}}{k_{24} k_{26}} k_{29} \frac{K_{3}}{K_{4}} \tau_{13}\\
a_{13} &:= d_a = \frac{k_{21}}{k_{20}} \frac{k_{25}+k_{26}}{k_{24} k_{26}} k_{29} \frac{K_{3}}{K_{4}} \tau_{13}\\
a_{14} &:= d_{an} = \frac{k_{25}+k_{26}}{k_{24} k_{26}} k_{29} \frac{K_{3}}{K_{4}} \tau_{13}\\
a_{15} &:=y_{in} = \frac{k_{28}+k_{29}}{k_{27}}\frac{\tau_{13}}{\tau_{12}}\\
a_{16} &:=p = 
\frac{k_{21}}{k_{20}}\frac{k_{22}}{k_{23}} \frac{k_{29}}{k_{35}}\frac{k_{27}}{k_{24} k_{26}}
\frac{k_{25}+k_{26}}{k_{28}+k_{29}}\frac{k_{30} k_{32}}{k_{33}}
\frac{k_{34}+k_{35}}{k_{31}+k_{32}}
\left(\frac{K_1 K_{3}}{K_2 K_{4}} \right)\tau_{12}
\\
a_{17} &:=p_n = \tau_{12}\\
a_{18} &:= c_{YD} = 
\frac{k_{21}}{k_{20}} \frac{k_{29} k_{30}}{k_{24} k_{26}} \frac{k_{25} + k_{26}}{k_{31} + k_{32}}
\left(\frac{K_1 K_{3}}{K_2 K_{4}} \right)\tau_{13}
\\
a_{19} &:=c_{YDn} = \frac{k_{29}}{k_{26}} \tau_{13}\\
a_{20} &:=c_{YP} = \frac{k_{32}}{k_{35}}
\frac{k_{21}}{k_{20}} \frac{k_{29}k_{30} }{k_{24} k_{26}}
\frac{k_{25} + k_{26}}{k_{31} + k_{32}}
\left(\frac{K_1 K_{3}}{K_2 K_{4}} \right)\tau_{13}
\\
a_{21} &:= c_{YPn} = \tau_{13}
\end{align*}
where
\allowdisplaybreaks
\begin{align*}
K_1 &= (k_{5}+k_{36})\sigma_1 k_{6} k_{1} (k_{2} k_{7} + k_{2} k_{37} + \sigma_2 k_{37} + k_{39} k_{37} + k_{3} k_{7} + k_{39} k_{7}) \\
K_2 &= k_{4} k_{6} k_{1}(k_{5}+k_{36}) (k_{37} \sigma_2+(k_{3}+k_{39})(k_{7}+k_{37})) \\
K_3 &= k_{4} k_{6} k_{2} k_{1} (k_{5}+k_{36})(k_{7}+k_{37}) \\
K_4 &= k_{4} k_{2} \sigma_2 k_{1} (k_{5}+k_{36})(k_{7}+k_{37})\\
K_5 &= k_{4} \sigma_1 k_{6} k_{1} (k_{2} k_{7} + k_{2} k_{37} + \sigma_2 k_{37} + k_{39} k_{37} + k_{3} k_{7} + k_{39} k_{7})\\
K_6 &= k_{4} k_{6}\left(k_{36} \sigma_1 (k_{7} k_{2}+k_{7} k_{3}+k_{37} k_{2}+k_{7} k_{39}+k_{37} k_{39}+k_{37} \sigma_2)\right.\\
&+(k_{5}+k_{36})(k_{37} \sigma_2 k_{2}+k_{7} k_{39} k_{2}+k_{37} k_{39} k_{2}+k_{37} \sigma_2 k_{38}+k_{7} k_{3} k_{38}\\
&+\left. k_{7} k_{39} k_{38}+k_{37} k_{3} k_{38}+k_{37} k_{39} k_{38})\right)\\
K_7 &= k_{4} k_{6} k_{2} \sigma_2 k_{1} (k_{5}+k_{36})
\end{align*}
and $\sigma_1, \sigma_2, d_{12}, \tau_{12}, \tau_{13} > 0$.

In the parametrizations of the positive equilibria of $\mathscr{N}_S$ and $\mathscr{N}_M$, no equilibrium concentration entirely depends on the rate constants $k_i$. In fact, all species concentrations depend on free parameters. Note that various combinations of these free parameters correspond to various sets of initial conditions. 
Hence, the two models do not have species that exhibit ACR. 

For $(ii)$, we can easily observe from the equilibria parametrization of Feinberg in Section \ref{eq:param:FAL} that only $a_{26}$ depends entirely on rate constants. Hence, the only ACR species in the Feinberg mass action system is $A_{26}$.
\end{proof}

\section{CORE analysis of the augmented Schmitz and the MacLean networks}

The CORE analysis in this section begins with the observation that the set of common reactions of Schmitz and MacLean is contained in the union of the FID subnetworks $\mathscr{N}_{S,1}$ $\cup$ $\mathscr{N}_{S,3}$ and $\mathscr{N}_{M,1}$ $\cup$ $\mathscr{N}_{M,3}$. This opens the possibility of relating any positive equilibria of the subnetwork generated by the set of common reactions to those of the two networks. 

\begin{remark}
    {We will use the notation $\langle \mathcal{R} \rangle$ to denote the reaction network generated by the set of reactions $\mathcal{R}$.}
\end{remark}

Let us define the augmented Schmitz network $\mathscr{N}_{SA}$ as the union of $\mathscr{N}_{S}$ and $\langle \{ A_4 \to 0, 0 \to A_{10}, 0 \to A_{11} \} \rangle$. $\mathscr{N}_{SA}$ shares all the basic structural and kinetic properties of $\mathscr{N}_S$, including its multi-stationarity. Their FIDs differ only in one subnetwork $\mathscr{N}_{SA,1}$, which is equal to $\mathscr{N}_{S,1}$ plus the three additional flow reactions. {Flow reactions are reactions of the form $A \to 0$ or $0 \to A$.}

We can now state the main result of this section:

\begin{theorem}
Let $\mathscr{N}_{SAM}:= \langle \mathscr{R}_{SA} \cap \mathscr{R}_{M} \rangle$ be the subnetwork of common reactions of the augmented Schmitz and MacLean networks. Then we have the following:
\begin{enumerate}
    \item[i.] $\mathscr{N}_{SAM}$ is a reversible and deficiency zero network.
    \item[ii.] The set of positive equilibria of $\mathscr{N}_{SAM}$  induces a subset of positive equilibria of $\mathscr{N}_{M}$.
    \item[iii.] A subset of positive equilibria of $\mathscr{N}_{SAM}$ induces a subset of positive equilibria of $\mathscr{N}_{SA}$.
\end{enumerate} 
\label{theorem:equilibria:networks}
\end{theorem}

\begin{proof}
    $(i)$ follows directly from Table \ref{tab:reactionsNSandNM}. $(ii)$ Since all the reactions of $\mathscr{N}_{SAM}$ are contained in $\mathscr{N}_{M,1}$ $\cup$ $\mathscr{N}_{M,2}$, then the stoichiometric subspace $S_{SAM}$ is a subspace of $S_{M,1}+S_{M,2}$. Since the rank of $\mathscr{N}_{SAM}$ is 5, which is equal to the rank of $\mathscr{N}_{M,1}$ $\cup$ $\mathscr{N}_{M,2}$, the spaces coincide. According to the Deficiency Zero Theorem (page 89 of \cite{FeinbergBook2019}), $\mathscr{N}_{SAM}$ has a unique, complex balanced, and stable equilibrium in each stoichiometric class.
    A result of Joshi and Shiu in \cite{JOSH2013} implies that the equilibria of $\mathscr{N}_{SAM}$ can be lifted to the equilibria of $\mathscr{N}_{M,1} \cup \mathscr{N}_{M,2}$. According to Lemma 3 of Lubenia et al. \cite{LUML2021}, such an equilibrium is an equilibrium for $\mathscr{N}_{M}$ if and only if there is an equilibrium of the complementary subnetwork $\mathscr{N}_{M,3} \cup \ldots \cup \mathscr{N}_{M,7}$  whose components in the common species with $\mathscr{N}_{M,1} \cup \mathscr{N}_{M,2}$ coincide with its components (the common species of the two independent subnetworks are $A_1, \ldots, A_5$).
    $(iii)$ Through the additional inflows, we obtain an independent decomposition $\mathscr{N}_{SA,1} \cup \mathscr{N}_{SA,2} = \mathscr{N}_{SAM} \cup \langle \{ A_{10} \rightleftarrows 0 \rightleftarrows A_{11} \} \rangle$
    so that the positive equilibria of $\mathscr{N}_{SAM}$ that are also the equilibria of $\langle \{ A_{10} \rightleftarrows 0 \rightleftarrows A_{11} \} \rangle$ are the equilibria of the union of FID subnetworks. The remaining argument is identical to that of $(ii)$.
\end{proof}

\begin{table}[ht!]
	\centering
	\caption{Species and reactions of the Schmitz ($\mathscr{N}_S$) and MacLean ($\mathscr{N}_M$) Wnt signaling models}
	\label{tab:reactionsNSandNM}
	\begin{tabular}{|l|l|}
		\hline
		\multicolumn{2}{|c|}{Common to $\mathscr{N}_S$ and $\mathscr{N}_M$} \\
            \hline
            \multicolumn{2}{|c|}{$A_1$, \dots, $A_9$} \\
		\hline
		\multicolumn{2}{|l|}{\hspace{2 cm} $R_1: 0 \rightarrow A_4$} \\
		\multicolumn{2}{|l|}{\hspace{2 cm} $R_2: A_4 \rightarrow A_5$} \\
		\multicolumn{2}{|l|}{\hspace{2 cm} $R_3: A_5 \rightarrow A_4$} \\
		\multicolumn{2}{|l|}{\hspace{2 cm} $R_4: A_1+A_4 \rightarrow A_8$} \\
		\multicolumn{2}{|l|}{\hspace{2 cm} $R_5: A_8 \rightarrow A_1 + A_4$} \\
		\multicolumn{2}{|l|}{\hspace{2 cm} $R_6: A_5+A_3 \rightarrow A_9$} \\
		\multicolumn{2}{|l|}{\hspace{2 cm} $R_7: A_9 \rightarrow A_5 + A_3$} \\
		\multicolumn{2}{|l|}{\hspace{2 cm} $R_8: A_6+A_5 \rightarrow A_7$} \\
		\multicolumn{2}{|l|}{\hspace{2 cm} $R_9: A_7 \rightarrow A_6 + A_5$} \\
		\hline
            \hline
		\multicolumn{1}{|c|}{Unique to $\mathscr{N}_S$} & \multicolumn{1}{c|}{Unique to $\mathscr{N}_M$} \\
            \hline
            \multicolumn{1}{|c|}{$A_{10}$, $A_{11}$} & \multicolumn{1}{c|}{$A_{12}$, \dots, $A_{21}$} \\
		\hline
		$R_{10}: A_8 \rightarrow A_1+A_{10}$ & $R_{18}: A_{12} \rightarrow A_{13}$ \\
		$R_{11}: A_9 \rightarrow A_3+A_{11}$ & $R_{19}: A_{13} \rightarrow A_{12}$ \\
		$R_{12}: A_{10} \rightarrow 0$ & $R_{20}: A_{13} \rightarrow A_{14}$ \\
		$R_{13}: A_{11} \rightarrow 0$ & $R_{21}: A_{14} \rightarrow A_{13}$ \\
		$R_{14}: A_1 \rightarrow A_2$ & $R_{22}: A_2 \rightarrow A_{15}$ \\
		$R_{15}: A_2 \rightarrow A_1$ & $R_{23}: A_{15} \rightarrow A_2$ \\
		$R_{16}: A_1 \rightarrow A_3$ & $R_{24}: A_3 + A_{14} \rightarrow A_{19}$ \\
		$R_{17}: A_3 \rightarrow A_1$ & $R_{25}: A_{19} \rightarrow A_3 + A_{14}$ \\
            & $R_{26}: A_{19} \rightarrow A_{14} + A_{15}$ \\
    	& $R_{27}: A_{15} + A_{17} \rightarrow A_{21}$ \\
    	& $R_{28}: A_{21} \rightarrow A_{15} + A_{17}$ \\
    	& $R_{29}: A_{21} \rightarrow A_3 + A_{17}$ \\
            & $R_{30}: A_{13} + A_1 \rightarrow A_{18}$ \\
    	& $R_{31}: A_{18} \rightarrow A_{13} +A_1$ \\
    	& $R_{32}: A_{18} \rightarrow A_{13} +A_2$ \\
    	& $R_{33}: A_2 + A_{16} \rightarrow A_{20}$ \\
    	& $R_{34}: A_{20} \rightarrow A_2 + A_{16}$ \\
    	& $R_{35}: A_{20} \rightarrow A_1 + A_{16}$ \\
    	& $R_{36}: A_8 \rightarrow A_1$ \\
    	& $R_{37}: A_9 \rightarrow A_3$ \\
    	& $R_{38}: A_4 \rightarrow 0$ \\
    	& $R_{39}: A_5 \rightarrow 0$ \\
		\hline
	\end{tabular}
\end{table}

\begin{remark}
\begin{enumerate}
    \item The additional reactions in the augmented Schmitz network do not change the stoichiometric subspaces since their reaction vectors are already contained in the stoichiometric subspace of the Schmitz network.
    \item The mass action kinetics considered in $(ii)$ and $(iii)$ are identical except in their {respective domains}.
    \item Statement $(ii)$ easily generalizes to positive dependent subnetworks comprised of unions of FID subnetworks (which are not the whole network). 
    \item The results in Theorem \ref{theorem:equilibria:networks} provide a further example of the usefulness of FIDs in comparative network analysis.
\end{enumerate}
\end{remark}

{\section{CORE analysis of the Feinberg and MacLean networks}}


Table \ref{tab:reactionsNFandNM} shows that the set of common reactions of $\mathscr{N}_F$ and $\mathscr{N}_M$ is as follows: $\{ R_1, R_4, R_5, R_{18}, R_{19}, R_{38} \}$. The following proposition describes the properties of the corresponding subnetwork.

\begin{table}[ht!]
	\centering
	\caption{Species and reactions of the Feinberg ($\mathscr{N}_F$) and MacLean ($\mathscr{N}_M$) Wnt signaling models}
	\label{tab:reactionsNFandNM}
	\begin{tabular}{|l|l|}
		\hline
		\multicolumn{2}{|c|}{Common to $\mathscr{N}_F$ and $\mathscr{N}_M$} \\
            \hline
            \multicolumn{2}{|c|}{$A_1$, $A_2$, $A_4$ $A_6$, $A_7$, $A_8$, $A_{12}$, $A_{13}$} \\
		\hline
		\multicolumn{2}{|l|}{\hspace{2 cm} $R_1: 0 \rightarrow A_4$} \\
		\multicolumn{2}{|l|}{\hspace{2 cm} $R_4: A_1+A_4 \rightarrow A_8$} \\
		\multicolumn{2}{|l|}{\hspace{2 cm} $R_5: A_8 \rightarrow A_1 + A_4$} \\
		\multicolumn{2}{|l|}{\hspace{2 cm} $R_{18}: A_{12} \rightarrow A_{13}$} \\
		\multicolumn{2}{|l|}{\hspace{2 cm} $R_{19}: A_{13} \rightarrow A_{12}$} \\
		\multicolumn{2}{|l|}{\hspace{2 cm} $R_{38}: A_4 \rightarrow 0$} \\
		\hline
            \hline
		\multicolumn{1}{|c|}{Unique to $\mathscr{N}_F$} & \multicolumn{1}{c|}{Unique to $\mathscr{N}_M$} \\
            \hline
            \multicolumn{1}{|c|}{$A_{10}$, $A_{23}$, \dots, $A_{28}$} & \multicolumn{1}{c|}{$A_3$, $A_5$, $A_9$, $A_{14}$, \dots, $A_{21}$} \\
		\hline
        $R_{12}: A_{10} \rightarrow 0$ & $R_{2}: A_{4} \rightarrow A_{5}$ \\
		$R_{14}: A_1 \rightarrow A_2$ & $R_{3}: A_{5} \rightarrow A_{4}$ \\
		$R_{15}: A_{2} \rightarrow A_1$ & $R_{6}: A_{5} +A_3 \rightarrow A_{9}$ \\
		$R_{43}: A_{24} + A_{26} \rightarrow A_{23}$ & $R_{7}: A_{9} \rightarrow A_{5}+A_3$ \\
		$R_{44}: A_{23} \rightarrow A_{24} + A_{26}$ & $R_{8}: A_6 +A_5 \rightarrow A_{7}$ \\
		$R_{45}: A_8 \rightarrow A_{25}$ & $R_{9}: A_{7} \rightarrow A_6 + A_5$ \\
		$R_{46}: A_{25} \rightarrow A_1 + A_{10}$ & $R_{20}: A_{13} \rightarrow A_{14}$ \\
		$R_{47}: 0 \rightarrow A_{26}$ & $R_{21}: A_{14} \rightarrow A_{13}$ \\
		$R_{48}: A_{26} \rightarrow 0$ & $R_{22}: A_2 \rightarrow A_{15}$ \\
		$R_{49}: A_4 + A_6 \rightarrow A_7$ & $R_{23}: A_{15} \rightarrow A_2$ \\
		$R_{50}: A_7 \rightarrow A_4 + A_6$ & $R_{24}: A_3 + A_{14} \rightarrow A_{19}$ \\
		$R_{51}: A_{24} + A_4 \rightarrow A_{27}$ & $R_{25}: A_{19} \rightarrow A_3 + A_{14}$ \\
        $R_{52}: A_{27} \rightarrow A_{24} + A_4$ & $R_{26}: A_{19} \rightarrow A_{14} + A_{15}$ \\
        $R_{53}: A_{13} + A_2 \rightarrow A_{28}$
    	& $R_{27}: A_{15} + A_{17} \rightarrow A_{21}$ \\
     $R_{54}: A_2 \rightarrow A_{23}$
    	& $R_{28}: A_{21} \rightarrow A_{15} + A_{17}$ \\
     $R_{55}: A_{23} \rightarrow A_2$
    	& $R_{29}: A_{21} \rightarrow A_3 + A_{17}$ \\
     $R_{56}: A_{28} \rightarrow A_{13} + A_{23}$
            & $R_{30}: A_{13} + A_1 \rightarrow A_{18}$ \\
    	& $R_{31}: A_{18} \rightarrow A_{13} +A_1$ \\
    	& $R_{32}: A_{18} \rightarrow A_{13} +A_2$ \\
    	& $R_{33}: A_2 + A_{16} \rightarrow A_{20}$ \\
    	& $R_{34}: A_{20} \rightarrow A_2 + A_{16}$ \\
    	& $R_{35}: A_{20} \rightarrow A_1 + A_{16}$ \\
    	& $R_{36}: A_8 \rightarrow A_1$ \\
    	& $R_{37}: A_9 \rightarrow A_3$ \\
    	& $R_{39}: A_5 \rightarrow 0$ \\
		\hline
	\end{tabular}
\end{table}

\begin{theorem}
    Let $\mathscr{N}_{FM} = \langle \{ R_1, R_4, R_5, R_{18}, R_{19}, R_{38} \} \rangle$ be the subnetwork of common reactions of $\mathscr{N}_F$ and $\mathscr{N}_M$. Then
    \begin{enumerate}
        \item[i.] $\mathscr{N}_{FM}$ is a reversible, deficiency zero subnetwork;
        \item[ii.] $\mathscr{N}_{FM}$ is contained in $\mathscr{N}_{F,1} \cup \mathscr{N}_{F,3}$, but is a dependent subnetwork of it; and
        \item[iii.] $\mathscr{N}_{FM}$ is contained in $\mathscr{N}_{M,1} \cup \mathscr{N}_{M,3}$, but is a dependent subnetwork of it.
    \end{enumerate}
\end{theorem}

\begin{proof}


    $(i)$ is evident from Table \ref{tab:reactionsNFandNM}.
    $(ii)$ As shown in Section \ref{section:comparison:three:models}, we have $\{ R_1, R_4, R_5, R_{38} \} \subset \mathscr{N}_{F,1}$ and $\{ R_{18}, R_{19} \} = \mathscr{N}_{F,3}$. However, the rank of $\mathscr{N}_{F,1} \cup \mathscr{N}_{F,3}$ is $4 + 1 = 5$ while the sum of the rank of $\mathscr{N}_{FM}$ and the rank of $(\mathscr{N}_{F,1} \cup \mathscr{N}_{F,3}) \backslash \mathscr{N}_{FM}$ is $3 + 3 = 6$. $(iii)$ Similarly, as shown in Section \ref{section:comparison:three:models}, $\{ R_1, R_4, R_5, R_{38} \} \subset \mathscr{N}_{M,1}$ and $\{ R_{18}, R_{19} \} = \mathscr{N}_{M,3}$. However, the rank of $\mathscr{N}_{M,1} \cup \mathscr{N}_{M,3}$ is $4 + 1 = 5$ while the sum of the rank of $\mathscr{N}_{FM}$ and the rank of $(\mathscr{N}_{M,1} \cup \mathscr{N}_{M,3}) \backslash \mathscr{N}_{FM}$ is $3 + 4 = 7$.
\end{proof}

In contrast to the equilibria of $\mathscr{N}_{SAM}$ being building blocks of the equilibria of $\mathscr{N}_{SA}$ and $\mathscr{N}_M$ (see the remark at the end of Section \ref{structural:property:FID}), the equilibria of $\mathscr{N}_{FM}$ are not building blocks of the equilibria of $\mathscr{N}_F$ and $\mathscr{N}_M$. 

\section{Concordance profile (CP) analysis of the Wnt networks}

Feinberg's multi-stationary network $\mathscr{N}_F$ has the remarkable property that removing a single reversible pair results in a concordant subnetwork. Concordance is a strong form of injectivity, and hence he dubbed this abrupt change of stability-related system properties an indictation of ``the subtlety of concordance''. The question of whether the other Wnt networks also possessed a high level of ``hidden concordance'' led us to develop concepts and procedures to quantify and measure this property using the network's FID and collect them in a method called Concordance Profile (CP) analysis. After a brief review of the concordance concepts and basic results, we introduce CP analysis and apply it to the Wnt reaction networks.

Concordant networks were introduced by G. Shinar and M. Feinberg \cite{SHFE2012} as an abstraction of continuous flow stirred tank reactors, a model that is extensively used in chemical engineering. 
The concept of concordance of a network, as described by them, refers to ``architectures that enforce duller and more restrictive behavior despite what might be the intricacy in the interplay of many species, even independent of the values that the kinetic parameters might take''. 
For detailed technical information on concordance, please refer to Appendix \ref{concordance:concept:properties}.

\subsection{Concordance dimension and concordance level of a reaction network}

The observation that Feinberg's network decomposed into the independent concordant subnetwork and the reversible pair pointed us to a network's FID as the starting point of our theory. After applying the Concordance Test of the CRNToolbox to each FID subnetwork, we define two subsets:

\begin{definition}
    The \emph{concordance (discordance) set} FIDC (FIDD) of the network is the set of all concordant (discordant) subnetworks of the FID. 
\end{definition}

Note that the FIDC or the FIDD can be empty.  For example, the network $\{ A_1 + 2 A_2 \rightleftarrows A_2 + 2 A_1 \}$ is discordant, and its FID is the trivial decomposition. Thus, its FIDC is $\varnothing$. On the other hand, the network $\{ A_2 \rightleftarrows A_3 \}$ is concordant and also has the trivial FID. Hence, its FIDD is $\varnothing$. Furthermore, the union of concordant networks need not be concordant even in an independent decomposition. This is shown by the Feinberg Wnt network.

If there is at least one concordant FID subnetwork, i.e., the FIDC is non-empty, we can define the following concepts:

\begin{definition}
    The concordance dimension (denoted by $c$) of a reaction network is the rank of a maximal independent concordant subnetwork. Such a subnetwork is called a concordance core. The number $d := s – c$ is called the discordance dimension of the network. The ratios $\frac{c}{s}$ and $\frac{d}{s}$ are called the concordance level and the discordance level of the network, respectively.
\end{definition}

The following remark collects basic properties and additional details about the concepts above:

\begin{remark}
    \begin{enumerate}
        \item If the concordance set (discordance set) is empty, we set $c = 0$ ($d = 0$).
        \item For any network, $0 \leq c$ and $d \leq s$.
        \item If the FIDC is non-empty, $c \geq \max (\mathrm{rank}(N_i))$, $N_i$ in the FIDC.
        \item There may be more than one concordance core in a reaction network. For example, the 2-site phosphorylation/dephosphorylation network \cite{CODH2018,Villareal2024} has two concordant cores of dimension $c = 1$ with $s = 2$.
    \end{enumerate}
\end{remark}


\begin{example}
    \begin{enumerate}
        \item[a.] For any concordant network, $c = s$, hence $\frac{c}{s} = 1$.
        \item[b.] For the Feinberg Wnt network, $c = 11$ while $s = 12$.  Hence, $\frac{c}{s} = 0.91$.
    \end{enumerate}
\end{example}

\subsection{CP analysis of the Wnt networks}

\subsubsection{Schmitz Wnt signaling network}

The Concordance Report generated by the Windows application CRNToolbox shows that all the FID subnetworks of $\mathscr{N}_S$ are concordant. Thus, the FIDC of $\mathscr{N}_S$ is
equal to the FID.

Recall that $\mathscr{N}_S$ is a discordant network. Removing the rank 1 subnetwork $\mathscr{N}_{S,4}$ results to a concordant subnetwork. However, when we remove the other rank 1 subnetworks $\mathscr{N}_{S,2}$ or $\mathscr{N}_{S,3}$ from $\mathscr{N}_S$, the resulting subnetwork remains discordant. Thus, the concordance core of $\mathscr{N}_S$ is $CC = \mathscr{N}_{S,1} \cup \mathscr{N}_{S,2} \cup \mathscr{N}_{S,3}$. The rank of $CC$ is 8, i.e., the concordance dimension is $c = 8$. This results to a concordance level of $\frac{c}{s} = \frac{8}{9} \approx 0.89$.

Despite its discordance, $\mathscr{N}_S$ seems to have a high degree of ``hidden concordance'' as evidenced by the concordance level of about 0.89.

\subsubsection{MacLean Wnt signaling network}

Similar to $\mathscr{N}_S$, the Concordance Report from the CRNToolbox shows that all FID subnetworks of $\mathscr{N}_M$ are concordant. Thus,
the FIDC equals the FID.

$\mathscr{N}_M$ is a discordant network. Removing any of the rank 1 subnetworks from $\mathscr{N}_M$ results to a discordant subnetwork. We then try to remove a combination of two rank 1 subnetworks: all combinations result to a discordant subnetwork except the removal of $\mathscr{N}_{M,4} \cup \mathscr{N}_{M,5}$ from $\mathscr{N}_M$ which gives rise to a concordant subnetwork. Thus, the FIDC of $\mathscr{N}_M$ is $CC = \mathscr{N}_{M,1} \cup \mathscr{N}_{M,2} \cup \mathscr{N}_{M,3} \cup \mathscr{N}_{M,6} \cup \mathscr{N}_{M,7}$. The rank of $CC$ is 12, i.e., the concordance dimension is $c = 12$. This results to a concordance level of $\frac{c}{s} = \frac{12}{14} \approx 0.86$.

Similar to $\mathscr{N}_S$, $\mathscr{N}_M$ has a high degree of ``hidden'' concordance (concordance level of about 0.86).

We conclude from CP analysis that the MacLean and Schmitz networks are more similar (concordance levels of 0.86 compared to 0.89) than the MacLean and Feinberg networks (0.86 versus 0.91).

\section{Summary and recommendation}
\label{sec:summary}

Through techniques such as the finest independent decompositions (FID), equilibria para- metrizations (EP), and our newly developed Common Reactions Equilibria (CORE) and Concordance Profile (CP) analyses, we have revealed significant findings in comparing reaction networks with respect to their sets of equilibria. In particular, we applied these techniques to Wnt signaling models in biochemistry.
First, we explored absolute concentration robustness (ACR) and found that, although the Feinberg model exhibits ACR in a specific species, both the Schmitz and MacLean models lack this property.
Second, our analyses using FID and CORE revealed important relationships between the equilibria sets of the augmented Schmitz and MacLean models.
Furthermore, we detected subtle differences
between the equilibria sets of Feinberg and MacLean models, 
which were not evident in standard reaction network analysis.
Lastly, based on the concordance levels, CP analysis suggests that the MacLean and Schmitz networks are more similar than the MacLean and Feinberg networks.
Moving forward, these findings offer valuable insights into the behavior of Wnt signaling networks, and comparative analysis of biochemical models, in general. 

\section*{Acknowledgement}
BSH acknowledge the Personally Funded Research support from the Institute of Mathematics, College of Science, and the Office of the Vice Chancellor for Research and Development of the University of the Philippines Diliman. BSH was awarded the MacArthur \& Josefina De los Reyes Professorial Chair in Mathematics of the 2023 UPFI Inc. Professorial Chairs and Faculty Grants (OVCAA, UP Diliman) for this project.



\appendix
\section{Preliminaries}
\label{chap:preliminaries}

\subsection{Fundamentals of chemical reaction networks}

A \emph{chemical reaction network} (CRN)
is defined by a triple of nonempty and finite sets $\mathscr{N}=(\mathscr{S}, \mathscr{C}, \mathscr{R})$ with
    \begin{itemize}
        \item[a.] \emph{species set} $\mathscr{S} = \left\{A_1, A_2, \ldots, A_m \right\}$,
        \item[b.] \emph{complex set} $\mathscr{C} = \{C_1, C_2, \ldots, C_n\}$ consisting of non-negative linear combinations of the species, and
        \item[c.] \emph{reaction set} $\mathscr{R} = \{R_1, R_2, \ldots, R_r\} \subset \mathscr{C} \times \mathscr{C}$.
    \end{itemize}
    
We commonly represent a reaction $(y, y')$ as $y \to y'$. Here, $y$ is referred to as \emph{reactant complex} while $y'$ is termed \emph{product complex}. Moreover, \emph{reaction vector} of the reaction is the difference $y' - y$, which is a linear combination, probably nonnegative, of the species.

The linear subspace within the vector space $\mathbb{R}^m$, over $\mathbb{R}$, generated by all the reaction vectors in the CRN is identified as the \emph{stoichiometric subspace} of $\mathscr{N}$, i.e., $S = \mathrm{span}\{y' - y \in \mathbb{R}^m \mid y \rightarrow y' \in \mathscr{R}\}$. The \emph{stoichiometric matrix} of the network is an $m \times r$ matrix, where each column contains the coefficients of the associated species in the corresponding reaction vector linked to the respective reaction.

The \emph{deficiency} of a CRN is given by $\delta = n - \ell - s$ where $n$ is the number of complexes, $\ell$ is the number of connected components and $s$ is the rank of the stoichiometric matrix, which coincides with the dimension of the stoichiometric subspace of the network. Finally, a CRN is \emph{weakly reversible} if each reaction belongs to a directed cycle.

The \emph{linkage classes} are the connected components of a CRN when viewed as an undirected graph. We denote the number of linkage classes by $\ell$.
Furthermore, a linkage class is said to be \emph{strong linkage class} if for each pair $i$ and $j$, there is a directed path from complex $C_i$ to complex $C_j$, and vice versa. We denote the number of strong linkage classes by $s \ell$. A \emph{terminal strong linkage classes} is a maximal strongly connected subnetwork where there are no reactions from a complex in the subgraph to a complex outside the subnetwork. We denote the number of terminal strong linkage classes by $t$.

\subsection{Fundamentals of chemical kinetic systems}

To describe the dynamics of the evolution of the concentrations of species over time, a CRN is endowed with kinetics. Kinetics is defined as follows:

A \emph{kinetics} for a reaction network $\mathscr{N}=(\mathscr{S}, \mathscr{C}, \mathscr{R})$ is an assignment to
each reaction $y \to y' \in \mathscr{R}$ of a continuously differentiable {rate function} $K_{y\to y'}: \mathbb{R}^m_{{\geq} 0} \to \mathbb{R}_{\ge 0}$ such that the following positivity condition holds:
$K_{y\to y'}(c) > 0$ if and only if ${\rm{supp \ }} y \subset {\rm{supp \ }} c$, where ${\rm{supp \ }} y$ refers to the support of the vector $y$, i.e., the set of species with nonzero coefficient in $y$.
The pair $\left(\mathscr{N},K\right)$ is called a \emph{chemical kinetic system}.
In particular, a kinetics for a CRN $(\mathscr{S},\mathscr{C},\mathscr{R})$ is \emph{mass-action} if for each reaction $y\to y'$ (i.e., $[y_1,y_2,\ldots, y_m]^\top \to [y_1',y_2',\ldots, y_m']^\top$),
$$K_{y\to y'}(x)=k_{y\to y'}\prod _{i \in \mathscr{S}} x_i^{y_i}$$
for some $k_{y\to y'}>0$.

The \emph{species formation rate function} (SFRF) of a chemical reaction system $(\mathscr{N},K)$ is given by $f\left( x \right) = \displaystyle \sum\limits_{{y} \to {y'} \in \mathscr{R}} {{K_{{y} \to {y'}}}\left( x \right)\left( {{y'} - {y}} \right)}.$
Note that the SFRF can be written as
$f(x) = NK(x)$ where where $N$ represents the stoichiometric matrix of the network, and $K$ is the vector of rate functions.
The system of \emph{ordinary differential equations} (ODEs) for a chemical kinetic system is described by $\dfrac{{dx}}{{dt}} = f\left( x \right)$, where $x$ denotes a vector that represents the concentrations of species evolving over time.

A \emph{steady state} or \emph{equilibrium} is represented by a vector $c$ of species concentrations satisfying the condition $f(c)=0$. An \emph{equilibrium} is \emph{positive} when each concentration in the vector is greater than zero.
We denote the set of positive steady states of a chemical kinetic system $(\mathscr{N},K)$ by $E:=E_+(\mathscr{N},K)$.
Assuming that $f$ is a differentiable function, an equilibrium $x^*$ is \emph{degenerate} if $\text{Ker}(J_{x^*} (f)) \cap S \neq \{0\}$ where $J_{x^*} (f)$ is the Jacobian of $f$ evaluated at $x^*$. Otherwise, the equilibrium is \emph{non-degenerate}.

The reaction vectors of a CRN are \emph{positively dependent} if for each reaction $C_i \rightarrow C_j$ in the network, there exists a positive number $\alpha_{C_i \rightarrow C_j}$ such that
\begin{equation*}
	\sum_{C_i \rightarrow C_j} \alpha_{C_i \rightarrow C_j} (C_j - C_i) = 0.
\end{equation*}
A CRN with positively dependent reaction vectors is called \emph{positive dependent}.

A CRN \emph{admits multiple (positive) equilibria} or is characterized as \emph{multi-stationary} if positive rate constants exist such that the ODE system has more than one stoichiometrically-compatible equilibria.


\subsection{Decompositions of chemical reaction networks}

A CRN can be decomposed into \emph{subnetworks} \cite{Feinberg1987stability,FeinbergBook2019,HEDC2021}. by partitioning its reaction set into disjoint subsets.
If the rank of the stoichiometric matrix of the whole network (the \emph{stoichiometric matrix} is equal to the sum of the ranks of the stoichiometric matrices of its subnetworks, then the decomposition is \emph{independent}. In this case, the subnetworks are called \emph{independent subnetworks}.
The significance of independent decompositions is apparent through the following result by Martin Feinberg \cite{Feinberg1987stability,FeinbergBook2019}. 

\begin{theorem}
\label{theorem:independent}
    Let $\mathcal{N}$ be a CRN endowed with kinetics $K$ and let $\mathcal{N}$ be decomposed into independent subnetworks $N_1, \ N_2, \ldots,  N_n$ such that the rate functions of the reactions in $\mathcal{N}$ are also the rate functions of the reactions in the smaller independent subnetworks. Then the set of positive steady states of the whole network coincides with the intersection of the sets of positive steady states of the $n$ independent subnetworks, i.e.,
    \begin{align*}
        E = E_1 \cap E_2 \cap \ldots \cap E_n.
    \end{align*}
\end{theorem}

\subsection{Concordance concepts and basic properties}
\label{concordance:concept:properties}

To define the concept, we need to consider the linear map
$L:\mathbb{R}^r \to S$ (with $S$ as the stoichiometric subspace) where $L(\alpha)=\displaystyle \sum_{y \to y' \in \mathscr{R}}\alpha_{y \to y'}(y'-y)$.

\begin{definition}
    Let $\mathscr{R}$ be the reaction set of a reaction network. The reaction network is \emph{concordant} if there do not exist an $\alpha \in \ker{L}$ and a nonzero $\sigma \in S$ having the following properties:
    \begin{enumerate}
        \item For each $y \to y' \in \mathscr{R}$ such that $\alpha_{y\to y'} \ne 0$, ${\rm{supp}}({y})$ contains a species $A$ for which ${\rm{sgn}}{(\sigma_A)}={\rm{sgn}}{(\sigma_{y \to y'})}$ where $\sigma_A$ denotes the term in $\sigma$ involving the species $A$ and $\rm{sgn}$ is the signum function.
        \item For each $y \to y' \in \mathscr{R}$ such that $\alpha_{y\to y'} = 0$, either ${\sigma_A}=0$ for all $A \in {\rm{supp}}({y})$, or else ${\rm{supp}}({y})$ contains species $A$ and $A'$ for which ${\rm{sgn}}{(\sigma_A)}={\rm{sgn}}{(\sigma_{A'})}$, but not zero.
    \end{enumerate}
    A network that is not concordant is called \emph{discordant}.
\end{definition}

Concordance is closely associated with two classes of kinetics on a network: injective and weakly monotonic kinetics, as described in \cite{SHFE2012}.

\begin{definition}
    A kinetic system $(\mathscr{N},K)$ is \emph{injective} if, for each pair of distinct stoichiometrically compatible vectors $x^*,x^{**} \in \mathbb{R}_{\ge 0}^m$ ($m$ here is the number of species), at least one of which is positive,
    \begin{equation*}
	\sum_{y \rightarrow y' \in \mathscr{R}} K_{y \to y' }(x^{**}) (y' - y) \ne \sum_{y \rightarrow y' \in \mathscr{R}} K_{y \to y'}(x^{*}) (y' - y).
\end{equation*}
\end{definition}

An injective kinetic system is necessarily a monostationary system. It cannot admit two distinct stoichiometrically compatible equilibria, at least one of which is positive.

\begin{definition}
    A kinetics $K$ for a reaction network $\mathscr{N}$ is \emph{weakly monotonic}, if for each pair of vectors $x^*,x^{**} \in \mathbb{R}_{\ge 0}^m$, the following implications hold for each $y \to y' \in \mathscr{R}$ such that ${\rm{supp}}({y}) \subset {\rm{supp}}({x^*})$ and ${\rm{supp}}({y}) \subset {\rm{supp}}({x^{**}})$:
    \begin{enumerate}
        \item $K_{y \to y' }(x^{**})>K_{y \to y' }(x^{*})$ $\implies$ there exists a species $A \in {\rm{supp}}({y})$ with $x_A^{**}>x_A^{*}$.
        \item $K_{y \to y' }(x^{**})=K_{y \to y' }(x^{*})$ $\implies$ $x_A^{**}=x_A^{*}$ for all $A \in {\rm{supp}}({y})$ or else there are species $A,A' \in {\rm{supp}}({y})$ with $x_A^{**}>x_A^{*}$ and $x_{A'}^{**}<x_{A'}^{*}$.
    \end{enumerate}
    We say that a system $(\mathscr{N},K)$ is \emph{weakly monotonic} when its kinetics $K$ is weakly monotonic.
\end{definition}

Some examples of weakly monotonic kinetics systems are mass action systems and a class of power law systems where all kinetic orders are nonnegative called non-inhibitory kinetics, denoted by PL-NIK.

\section{Meaning of the species in Wnt signaling networks}
\label{app:var}

Table \ref{tab:species} provides the species in the Wnt signaling networks that we consider.

\begin{table}[ht!]
\centering
\caption{Species and corresponding biomolecules that can occur in the Wnt signaling models considered}
\label{tab:species}
\begin{tabular}{|c|l|}
    \hline
    Species & \multicolumn{1}{c|}{Meaning} \\
    \hline
    $A_1$ & destruction complex (DC) (active form) \\
    \hline
    $A_2$ & DC (inactive form) \\
    \hline
    $A_3$ & active DC residing in the nucleus \\
    \hline
    $A_4$ & $\beta$-catenin \\
    \hline
    $A_5$ & $\beta$-catenin in the nucleus \\
    \hline
    $A_6$ & T-cell factor (TCF) \\
    \hline
    $A_7$ & $\beta$-catenin-TCF complex \\
    \hline
    $A_8$ & $\beta$-catenin bound with DC \\
    \hline
    $A_9$ & $\beta$-catenin bound with DC in the nucleus \\
    \hline
    $A_{10}$ & $\beta$-catenin (for proteasomal degradation) \\
    \hline
    $A_{11}$ & $\beta$-catenin (for proteasomal degradation) in the nucleus \\
    \hline
    $A_{12}$ & dishevelled (inactive form) \\
    \hline
    $A_{13}$ & dishevelled (active form) \\
    \hline
    $A_{14}$ & active dishevelled in the nucleus \\
    \hline
    $A_{15}$ & inactive DC in the nucleus \\
    \hline
    $A_{16}$ & phosphatase \\
    \hline
    $A_{17}$ & phosphatase in the nucleus \\
    \hline
    $A_{18}$ & active DC bound with dishevelled \\
    \hline
    $A_{19}$ & active DC bound with dishevelled in the nucleus \\
    \hline
    $A_{20}$ &active DC bound with phosphatase \\
    \hline
    $A_{21}$ & active DC bound with phosphatase in the nucleus \\
    \hline
    $A_{22}$ & GSK3$\beta$ \\
    \hline
    $A_{23}$ & axin-APC complex \\
    \hline
    $A_{24}$ & APC \\
    \hline
    $A_{25}$ & $\beta$-catenin bound with DC (for proteasomal degradation) \\
    \hline
    $A_{26}$ & axin \\
    \hline
    $A_{27}$ & $\beta$-catenin-axin complex \\
    \hline
    $A_{28}$ & a complex considered as a single species (in \cite{FeinbergBook2019}): \\
    & ($A_{13}+A_{22}+A_{23}=A_{28}$)\\
    \hline
\end{tabular}
\end{table}

\end{document}